\documentclass[12pt,twoside]{article}

\usepackage{amsmath}
\usepackage{amssymb}
\usepackage{amsthm}
\usepackage[dvips]{graphicx}
\usepackage{type1cm}

\newtheorem{thm}{Theorem}[section]
\newtheorem{defi}[thm]{Definition}
\newtheorem{lem}[thm]{Lemma}
\newtheorem{prop}[thm]{Proposition}
\newtheorem{cor}[thm]{Corollary}
\newtheorem{rem}[thm]{Remark}

\title{\textbf{\Large Special Legendrian submanifolds in \\ 
toric Sasaki-Einstein manifolds}}

\author{\textsc{Takayuki Moriyama}}
\date{}

\begin{document}
\maketitle
\thispagestyle{empty}

\begin{quote}
\textbf{\fontsize{10pt}{15pt}\selectfont Abstract.}
\textrm{\fontsize{10pt}{15pt}\selectfont 
We show that every toric Sasaki-Einstein manifold $S$ admits 
a special Legendrian submanifold $L$ which arises as 
the link ${\rm fix}(\tau)\cap S$ of the fixed point set ${\rm fix}(\tau)$ 
of an anti-holomorphic involution $\tau$ on the cone $C(S)$. 
In particular, we obtain a special Legendrian torus $S^{1}\times S^{1}$ in 
an irregular toric Sasaki-Einstein manifold which is diffeomorphic to $S^{2}\times S^{3}$. 
Moreover, there exists a special Legendrian submanifold in 
$\sharp m(S^{2}\times S^{3})$ for each $m\ge 1$. 
}
\end{quote}

\section{Introduction}
A Sasaki-Einstein manifold is a $(2n+1)$-dimensional Riemannian manifold $(S,g)$ 
whose metric cone $(C(S),\overline{g})=(\mathbb{R}_{>0}\times S, dr^{2}+r^{2}g)$ 
is a Ricci-flat K\"{a}hler manifold where $r$ is the coordinate of $\mathbb{R}_{>0}$. 
We assume that $S$ is simply connected. 
Then the cone $C(S)$ is a complex $(n+1)$-dimensional Calabi-Yau manifold 
which admits a holomorphic $(n+1)$-form $\Omega$ and 
a K\"{a}hler form $\omega$ on $C(S)$ satisfying the Monge-Amp\`{e}re equation 
\begin{equation*}
\Omega\wedge \overline{\Omega}=c_{n+1}\omega^{n+1}
\end{equation*}
for a constant $c_{n+1}$. 
Then the real part $\Omega^{\rm Re}$ of $\Omega$ is a calibration 
whose calibrated submanifolds are called 
{\it special Lagrangian submanifolds}~\cite{HL}. 
An $n$-dimensional submanifold $L$ in a Sasaki-Einstein manifold $(S,g)$ is 
a {\it special Legendrian submanifold} if 
the cone $C(L)$ is a special Lagrangian submanifold in $C(S)$. 
We identify $S$ with the hypersurface $\{r=1\}$ in $C(S)$. 
Then $L$ is regarded as the link $C(L)\cap S$ of $C(L)$. 

Recently, toric Sasaki-Einstein manifolds have been constructed~\cite{BGN, FOW, GMSW, K}. 
The purpose of this paper is to construct a special Legendrian submanifold 
in every toric Sasaki-Einstein manifold. 
For a toric Sasaki manifold $(S,g)$, the metric cone $(C(S),\overline{g})$ is a toric K\"{a}hler variety. 
Then there exists an anti-holomorphic involution $\tau$ on $C(S)$. 
\begin{thm}\label{s1t1}
Let $(S,g)$ be a compact simply connected toric Sasaki-Einstein manifold. 
Then the link ${\rm fix}(\tau)\cap S$ is a special Legendrian submanifold. 
\end{thm}
The fixed point set of an isometric and anti-holomorphic involution is called the {\it real form}. 
It is well known that a real form of a Calabi-Yau manifold is a special Lagrangian submanifold. 
The point of Theorem~\ref{s1t1} is to show that the real form ${\rm fix}(\tau)$ 
arises as the cone of the link ${\rm fix}(\tau)\cap S$. 

A typical example of Sasaki-Einstein manifolds is 
the odd-dimensional unit sphere $S^{2n+1}$ with the standard metric, 
then the cone is the complex space $\mathbb{C}^{n+1}\backslash \{0\}$. 
Special Lagrangian cones in $\mathbb{C}^{n+1}\backslash \{0\}$ are regarded 
as special Lagrangian subvarieties in $\mathbb{C}^{n+1}$ 
with an isolated singularity at the origin. 
Joyce had provided the theory of special Lagrangian submanifolds 
in $\mathbb{C}^{n+1}$ with conical singularities \cite{J2}. 
Many examples of special Lagrangian submanifolds in $\mathbb{C}^{n+1}$ 
with the isolated singularity at the origin 
had been constructed~\cite{CM, H, J3, Mc}. 
These special Lagrangian cones induce 
special Legendrian submanifolds in the sphere $S^{2n+1}$. 
Recently, Haskins and Kapouleas gave a construction of 
special Legendrian immersions into the sphere $S^{2n+1}$ \cite{HK}. 
Special Legendrian submanifolds have also the aspect of minimal Legendrian submanifolds. 
On the sphere $S^{2n+1}$, the standard Sasaki-Einstein structure is regular 
and induced from the Hopf fibration $S^{2n+1}\to \mathbb{C}P^{n}$. 
Some special Legendrian submanifolds in $S^{2n+1}$ arise as 
lifts of minimal Lagrangian submanifolds in $\mathbb{C}P^{n}$~\cite{CM, Mc}. 
We have a generalization of Theorem~\ref{s1t1} as follows : 
\begin{thm}\label{s1t1.5}
Let $(S,g)$ be a compact toric Sasaki manifold. 
Then the link ${\rm fix}(\tau)\cap S$ is a totally geodesic Legendrian submanifold. 
\end{thm}

There exist two interesting points of our theorems. 
One is that we can construct a special Legendrian submanifold 
in every toric Sasaki-Einstein manifold 
which is not necessarily the sphere $S^{2n+1}$. 
The other is that some of these special Legendrian submanifolds are 
totally geodesic Legendrian submanifolds in irregular Sasaki-Einstein manifolds. 
A Sasaki-Einstein manifold of dimension $3$ is 
finitely covered by the standard $3$-sphere $S^{3}$. 
Hence we will consider the case of Sasaki-Einstein manifolds 
whose dimension are greater than or equal to $5$. 
Gauntlett, Martelli, Sparks and Wardram provided 
a family of explicit Sasaki-Einstein metrics $g_{p,q}$ on $S^{2}\times S^{3}$~\cite{GMSW}. 
Let $Y_{p,q}$ denote the Sasaki-Einstein manifold $(S^{2}\times S^{3}, g_{p,q})$. 
\begin{thm}\label{s1t2}
There exists a special Legendrian torus $S^{1}\times S^{1}$ 
in the toric Sasaki-Einstein manifold $Y_{p,q}$. 
\end{thm}

Any simply connected toric Sasaki-Einstein $5$-manifold is diffeomorphic to 
the $m$-fold connected sum $\sharp m(S^{2}\times S^{3})$ 
of $S^{2}\times S^{3}$ for an integer $m\ge 0$ 
where $\sharp m(S^{2}\times S^{3})$ for $m=0$ means the $5$-sphere $S^{5}$. 
Boyer, Galicki, Nakamaye and Koll\'{a}r showed that 
there exist many Sasaki-Einstein metrics on 
$\sharp m(S^{2}\times S^{3})$~\cite{BGN, K} for each $m\ge 1$. 
Van Covering provided a toric Sasaki-Einstein metric on 
$\sharp m(S^{2}\times S^{3})$ for each odd $m>1$~\cite{VC}. 
Futaki, Ono and Wang showed there exists a Sasaki-Einstein metric 
on a toric Sasaki manifold such that $c_{B}^{1}>0$ and $c^{1}(D)=0$~\cite{FOW}. 
Moreover, the uniqueness of Sasaki-Einstein metrics on toric Sasaki manifold 
was proven in \cite{CFO}. 
It implies that there exists an infinite inequivalent family of 
toric Sasaki-Einstein metrics on 
$S=\sharp m(S^{2}\times S^{3})$ for any $m \ge1$. 
Let $\tau$ be the anti-holomorphic involution on 
the toric K\"{a}hler cone $C(S)$ constructed in \S 3.3. 
Then Theorem \ref{s1t1} implies the following corollary :  
\begin{cor}
For any $m\ge 1$, the link ${\rm fix}(\tau)\cap S$ is a special Legendrian submanifold 
in $\sharp m(S^{2}\times S^{3})$. 
\end{cor}

The paper is organized as follows. 
In Section 2, we recall basic facts about Sasakian geometry. 
We introduce weighted Calabi-Yau structures 
on the K\"{a}hler cones of Sasaki manifolds 
which characterize Sasaki-Einstein structures on Sasaki manifolds. 
In Section 3, we define special Legendrian submanifolds in Sasaki-Einstein manifolds 
and provide a method to find special Legendrian submanifolds by 
considering the fixed point set of an anti-holomorphic involution. 
%There exists an intrinsic anti-holomorphic involution on 
%the cone of any toric Sasaki manifold. 
We apply the method to toric Sasaki-Einstein manifolds, and prove Theorem~\ref{s1t1}. 
We also provide Theorem~\ref{s1t1.5} as a generalization of Theorem~\ref{s1t1}. 
We show Theorem~\ref{s1t2} and give examples of special Legendrian submanifolds. 

\section{Sasakian geometry}
In this section, we will give a brief review of 
some elementary results in Sasakian geometry. 
For much of this material, 
we refer to~\cite{BG2} and \cite{S}. 
We assume that $S$ is a smooth manifold of dimension $(2n+1)$. 

\subsection{Sasaki structures}

\begin{defi}
{\rm 
A Riemannian manifold $(S,g)$ is a {\it Sasaki manifold} if and only if 
the metric cone $(C(S),\overline{g})=(\mathbb{R}_{>0}\times S, dr^{2}+r^{2}g)$ 
is K\"{a}hler for a complex structure. 
}
\end{defi}
We identify the manifold $S$ with the hypersurface $\{r=1\}$ of $C(S)$. 
Let $J$ and $\omega$ denote the complex structure 
and the K\"{a}hler form on the K\"{a}hler manifold $(C(S),\overline{g})$, respectively. 
The vector field $r\frac{\partial}{\partial r}$ is called 
the {\it Euler vector field} on $C(S)$. 
We define a vector field $\xi$ and a $1$-form $\eta$ on $C(S)$ by 
\begin{equation*}\label{s2.1eq1}
\xi=J(r\frac{\partial}{\partial r}),\quad \eta(X)=\frac{1}{r^{2}}\overline{g}(\xi, X)
\end{equation*}
for any vector field $X$ on $C(S)$. 
The vector field $\xi$ is a Killing vector field, i.e. $L_{\xi}\overline{g}=0$, and 
$\xi+\sqrt{-1}\,J\xi=\xi-\sqrt{-1}\,r\frac{\partial}{\partial r}$ is 
a holomorphic vector field on $C(S)$. 
%We call $\xi$ a {\it characteristic vector field} on $C(S)$. 
It follows from $L_{\xi}\eta=JL_{r\frac{\partial}{\partial r}}\eta=0$ that 
\begin{equation}\label{s2.1eq2}
\eta(\xi)=1,\quad i_{\xi}d\eta=0
\end{equation}
where $i_{\xi}$ means the interior product. 
The form $\eta$ is expressed as 
\begin{equation*}\label{s2.1eq3}
\eta=d^{c}\log r=\sqrt{-1}\,(\overline{\partial}-\partial)\log r
\end{equation*} 
where $d^{c}$ is the composition $-J\circ d$ of the exterior derivative $d$ and 
the action of the complex structure $-J$ on differential forms. 
We define an action $\lambda$ of $\mathbb{R}_{>0}$ on $C(S)$ by 
\begin{equation*}\label{s2.5eq9.5}
\lambda_{a}(r,x)=(ar,x)
\end{equation*}
for $a\in\mathbb{R}_{>0}$ and $(r,x)\in \mathbb{R}_{>0}\times S=C(S)$. 
If we put $a=e^{t}$ for $t\in\mathbb{R}$, 
then it follows from $L_{r\frac{\partial}{\partial r}}=\frac{d}{dt}\lambda_{e^{t}}^{*}|_{t=0}$ 
that $\{\lambda_{e^{t}}\}_{t\in\mathbb{R}}$ is one parameter group of transformations such that 
$r\frac{\partial}{\partial r}$ is the infinitesimal transformation. 
Then the K\"{a}hler form $\omega$ satisfies $\lambda_{a}^{*}\omega=a^{2}\omega$ for $a\in\mathbb{R}_{>0}$ and 
\begin{equation*}\label{s2.1eq4}
L_{r\frac{\partial}{\partial r}}\omega=2\omega.
\end{equation*}
It implies that  
\begin{equation}\label{s2.1eq5}
\omega=\frac{1}{2}d(r^{2}\eta)=\frac{\sqrt{-1}}{2}\partial\overline{\partial}r^{2}.
\end{equation}
Hence $\frac{1}{2}r^{2}$ is a K\"{a}hler potential on $C(S)$. 

The $1$-form $\eta$ induces the restriction $\eta|_{S}$ on $S\subset C(S)$. 
Since $L_{r\frac{\partial}{\partial r}}\eta=0$, 
the form $\eta$ is the extension of $\eta|_{S}$ to $C(S)$. 
The vector field $\xi$ is tangent to the hypersurface $\{r=c\}$ 
for each positive constant $c$. 
In particular, $\xi$ is considered as the vector field on $S$ 
and satisfies $g(\xi,\xi)=1$ and $L_{\xi}g=0$. 
Hence we shall not distinguish between $(\eta,\xi)$ on $C(S)$ and 
the restriction $(\eta|_{S},\xi|_{S})$ on $S$. 
Then the form $\eta$ is a contact $1$-form on $S$ : 
\begin{equation*}\label{s2.1eq6}
\eta\wedge (d\eta)^{n}\neq 0
\end{equation*}
since $\omega$ is non-degenerate. 
The equation (\ref{s2.1eq2}) implies that 
\begin{equation}\label{s2.1eq7}
\eta(\xi)=1,\quad i_{\xi}d\eta=0
\end{equation}
on $S$. 
For a contact form $\eta$, a vector field $\xi$ on $S$ 
satisfying the equation (\ref{s2.1eq7})  
is unique, and called the {\it Reeb vector field}. 
We define the contact subbundle $D\subset TS$ by $D=\ker\eta$. 
Then the tangent bundle $TS$ has the orthogonal decomposition 
\begin{equation*}\label{s2.1eq8}
TS=D\oplus \langle \xi\rangle
\end{equation*}
where $\langle \xi\rangle$ is the line bundle generated by $\xi$. 
We define a section $\Phi$ of ${\rm End}(TS)$ by setting 
$\Phi |_{D}=J|_{D}$ and $\Phi |_{\langle \xi\rangle}=0$. 
One can see that 
\begin{eqnarray}
\Phi^{2}=-{\rm id}+\xi\otimes\eta, \label{s2.1eq9}\\ 
d\eta(\Phi X,\Phi Y)=d\eta(X,Y) \label{s2.1eq10}
\end{eqnarray}
for any $X,Y \in TS$. 
Then the Riemannian metric $g$ satisfies 
\begin{equation}\label{s2.1eq11}
g(X,\Phi Y)=d\eta(X,Y)
\end{equation}
for any $X,Y \in TS$. 

We say a data $(\xi,\eta,\Phi,g)$ a \textit{contact metric structure} on $S$ 
if for a contact form $\eta$ and a Reeb vector field $\xi$, 
a section $\Phi$ of ${\rm End}(TS)$ and a Riemannian metric $g$ satisfy 
the equations (\ref{s2.1eq9}), (\ref{s2.1eq10}) and (\ref{s2.1eq11}). 
Moreover, a contact metric structure $(\xi,\eta,\Phi,g)$ is called 
a \textit{K-contact structure} on $S$ if $\xi$ is a Killing vector field with respect to $g$. 
The section $\Phi$ of a K-contact structure $(\xi,\eta,\Phi,g)$ defines 
an almost CR structure $(D,\Phi|_{D})$ on $S$. 
As we saw above, any Sasaki manifold $(S,g)$ has 
a K-contact structure $(\xi,\eta,\Phi,g)$ with 
the integrable CR structure $(D,\Phi|_{D}=J|_{D})$ on $S$. 
Conversely, if we have such a structure $(\xi,\eta,\Phi,g)$ on $S$, 
then $(\overline{g},\frac{1}{2}d(r^{2}\eta))$ is a K\"{a}hler structure 
on the cone $C(S)$, hence $(S,g)$ is a Sasaki manifold. 
We call a K-contact structure $(\xi,\eta,\Phi,g)$ with 
the integrable CR structure $(D,\Phi|_{D})$ a \textit{Sasaki structure} on $S$.

\subsection{The Reeb foliation}
Let $(\xi,\eta,\Phi,g)$ be a Sasaki structure on $S$. 
Then the Reeb vector field $\xi$ generates a foliation $\mathcal{F}_{\xi}$ of codimension $2n$ on $S$. 
The foliation $\mathcal{F}_{\xi}$ is called a {\it Reeb foliation}. 
A Reeb foliation $\mathcal{F}_{\xi}$ is {\it quasi-regular} 
if any orbit of the Reeb vector field $\xi$ is compact. 
Then each orbit is associated with a locally free $S^{1}$ action. 
If the $S^{1}$ action is free, $\mathcal{F}_{\xi}$ is called {\it regular}. 
If $\mathcal{F}_{\xi}$ is not quasi-regular, it is called {\it irregular}.  

A differential form $\phi$ on $S$ is called {\it basic} if 
\begin{equation*}\label{s2.2eq1}
i_{v}\phi=0,\quad  L_{v}\phi=0
\end{equation*}
for any $v\in\Gamma(\langle \xi\rangle)$. 
Let $\wedge_{B}^{k}$ be the sheaf of basic $k$-forms 
on the foliated manifold $(S,\mathcal{F}_{\xi})$. 
It is easy to see that for a basic form $\phi$ the derivative $d\phi$ is also basic. 
Thus the exterior derivative $d$ induces 
the operator $d_{B}=d|_{\wedge_{B}^{k}}:\wedge_{B}^{k}\to\wedge_{B}^{k+1}$ by the restriction. 
The corresponding complex $(\wedge_{B}^{*},d_{B})$ associates 
the cohomology group $H_{B}^{*}(S)$ which is called the {\it basic de Rham cohomology group}. 
If $\mathcal{F}_{\xi}$ is a transversely holomorphic foliation 
(see the next section for the definition), 
the associate transverse complex structure $I$ on $(S,\mathcal{F}_{\xi})$ 
give rises to the decomposition 
$\wedge^{k}_{B}\otimes\mathbb{C}=\oplus_{r+s=k}\wedge^{r,s}_{B}$ 
in the same manner as complex geometry, 
and we have the operators 
\begin{eqnarray*}
\partial_{B}:\wedge_{B}^{p,q}\to\wedge_{B}^{p+1,q} \\ 
\overline{\partial}_{B}:\wedge_{B}^{p,q}\to\wedge_{B}^{p,q+1}.
\end{eqnarray*}
We denote by $H_{B}^{p,*}(S)$ the cohomology of 
the complex $(\wedge_{B}^{p,*},\overline{\partial}_{B})$ 
which is called the {\it basic Dolbeault cohomology group}. 
%The transverse K\"{a}hler form $\omega^{T}$ defines 
%the $(1,1)$-basic Dolbeault cohomology class $[\omega^{T}]\in H_{B}^{1,1}(S)$. 

On the cone $C(S)$, a foliation $\mathcal{F}_{\langle \xi,r\frac{\partial}{\partial r}\rangle}$ 
is induced by the vector bundle $\langle \xi,r\frac{\partial}{\partial r}\rangle$ 
generated by $\xi$ and $r\frac{\partial}{\partial r}$. 
%The complex vector field $\xi+\sqrt{-1}r\frac{\partial}{\partial r}$ is 
%a holomorphic vector field which generates the $\mathbb{C}^{*}$-action. 
%Hence $\mathcal{F}_{\langle \xi,r\frac{\partial}{\partial r}\rangle}$ is a holomorphic foliation. 
Let $\widetilde{\phi}$ be a basic form  on 
$(C(S), \mathcal{F}_{\langle \xi, r\frac{\partial}{\partial r}\rangle})$, 
that is, $i_{v}\phi=L_{v}\phi=0$ for any $v\in\Gamma(\langle\xi, r\frac{\partial}{\partial r}\rangle)$. 
Then the restriction $\widetilde{\phi}|_{S}$ of $\phi$ to $S$ is also basic on $(S,\mathcal{F}_{\xi})$. 
Conversely, for any basic form $\phi$ on $(S,\mathcal{F}_{\xi})$, 
the trivial extension $\widetilde{\phi}$ of $\phi$ to $C(S)=\mathbb{R}_{>0}\times S$ is 
a basic form on $(C(S),\mathcal{F}_{\langle \xi,r\frac{\partial}{\partial r}\rangle})$. 
In this paper, we identify a basic form $\phi$ on $(S,\mathcal{F}_{\xi})$
with the extension $\widetilde{\phi}$ on 
$(C(S), \mathcal{F}_{\langle \xi,r\frac{\partial}{\partial r}\rangle})$. 
%In particular, we do not distinguish a basic function on $(S,\mathcal{F}_{\xi})$ from the extension to $C(S)$. 

\subsection{Transverse K\"{a}hler structures}
Let $\mathcal{F}$ be a foliation of codimension $2n$ on $S$. 
Then there exists a system $\{U_{i},f_{i},\gamma_{ij}\}$ 
consisting of an open covering $\{U_{i}\}_{i}$ of $S$, 
submersions $f_{i}:U_{i}\to \mathbb{C}^{n}$ 
and diffeomorphisms $\gamma_{ij}:f_{i}(U_{i}\cap U_{j})\to f_{j}(U_{i}\cap U_{j})$ 
for $U_{i}\cap U_{j}\neq \phi$ satisfying $f_{j}=\gamma_{ij}\circ f_{i}$ 
such that any leaf of $\mathcal{F}$ is given by each fiber of $f_{i}$. 
The foliation $\mathcal{F}$ is a \textit{transverse holomorphic foliation} 
(resp. a \textit{transverse K\"{a}hler foliation}) if 
there exists a system $\{U_{i},f_{i},\gamma_{ij}\}$ such that 
$\gamma_{ij}$ is bi-holomorphic (resp. preserving a K\"{a}hler structure) of $\mathbb{C}^{n}$. 
In order to characterize transverse structures on $(S,\mathcal{F})$, 
we consider the quotient bundle $Q=TS/F$ 
where $F$ is the line bundle associated by the foliation $\mathcal{F}$. 
We define an action of $\Gamma(F)$ 
to any section $I\in \Gamma({\rm End}(Q))$ as follows : 
\begin{equation*}\label{s2.3eq1}
(L_{v}I)(u)=L_{v}(I(u))-I(L_{v}u)
\end{equation*}
for $v\in\Gamma(F)$ and $u\in\Gamma(Q)$. 
If $I$ is a complex structure of $Q$, i.e. $I^{2}=-{\rm id}_{Q}$, 
and satisfies that $L_{v}I=0$ for any $v\in\Gamma(F)$, 
then a tensor $N_{I}\in \Gamma(\otimes^{2}Q^{*}\otimes Q)$ can be defined by 
\begin{equation*}\label{s2.3eq2}
N_{I}(u,w)=[Iu,Iw]_{Q}-[u,w]_{Q}-I[u,Iw]_{Q}-I[Iu,w]_{Q}
\end{equation*}
for $u,w\in\Gamma(Q)$, 
where $[u,w]_{Q}$ denotes the bracket $\pi[\widetilde{u},\widetilde{w}]$ 
for each lift $\widetilde{u}$ and $\widetilde{w}$ by the quotient map $\pi:TS\to Q$. 
\begin{defi}
{\rm 
A section $I\in \Gamma({\rm End}(Q))$ is 
a \textit{transverse complex structure on $(S,\mathcal{F})$} 
if $I$ is a complex structure of $Q$ such that $L_{v}I=0$ for any $v\in\Gamma(F)$ and $N_{I}=0$. 
}
\end{defi}
A foliation $\mathcal{F}$ is transversely holomorphic 
if and only if there exists a transverse complex structure $I$ on $(S,\mathcal{F})$. 
If a basic $2$-form $\omega^{T}$ satisfies $d\omega^{T}=0$ and $(\omega^{T})^{n}\neq 0$, 
then we call the form $\omega^{T}$ 
a \textit{transverse symplectic structure} on $(S,\mathcal{F})$. 
We can consider the basic form $\omega^{T}$ as a tensor of $\wedge^{2}Q^{*}$. 
Then the pair $(\omega^{T},I)$ is called 
a \textit{transverse K\"{a}hler structure} on $(S,\mathcal{F})$ 
if the $2$-tensor $\omega^{T}(\cdot,I\cdot)$ is positive on $Q$ 
and $\omega^{T}(I\cdot,I\cdot)=\omega^{T}(\cdot,\cdot)$ holds. 
Then we define the $2$-tensor $g^{T}$ by 
$g^{T}(\cdot,\cdot)=\omega^{T}(\cdot,I\cdot)$ and 
call it a {\it transverse K\"{a}hler metric} on $(S,\mathcal{F})$. 

Let $(\xi,\eta,\Phi,g)$ be a Sasaki structure and $\mathcal{F}_{\xi}$ the Reeb foliation on $S$. 
We can consider $\Phi$ as a section of ${\rm End}(Q)$ since $\Phi|_{\langle \xi\rangle}=0$. 
Then $\Phi$ is a transverse complex structure on $(S,\mathcal{F}_{\xi})$ 
by the integrability of the CR structure $\Phi|_{D}$. 
Moreover, the pair $(\Phi,\frac{1}{2}d\eta)$ is a transverse K\"{a}hler structure 
with the transverse K\"{a}hler metric $g^{T}(\cdot,\cdot)=\frac{1}{2}d\eta(\cdot,\Phi\cdot)$ 
on $(S,\mathcal{F}_{\xi})$. 
We define ${\rm Ric}^{T}$ as the Ricci tensor of $g^{T}$ 
which is called the transverse Ricci tensor. 
The transverse Ricci form $\rho^{T}$ is defined 
by $\rho^{T}(\cdot,\cdot)={\rm Ric}^{T}(\cdot,\Phi\cdot)$. 
The form $\rho^{T}$ is a basic $d$-closed $(1,1)$-form on $(S,\mathcal{F}_{\xi})$ 
and defines a $(1,1)$-basic Dolbeault cohomology class $[\rho^{T}]\in H_{B}^{1,1}(S)$ 
as in the K\"{a}hler case. 
The basic class $[\frac{1}{2\pi}\rho^{T}]$ in $H_{B}^{1,1}(S)$ is called 
the {\it basic first Chern class} on $(S,\mathcal{F}_{\xi})$ 
and is denoted by $c_{1}^{B}(S)$ (for short, we write it $c_{1}^{B}$). 
We say the basic first Chern class is positive (resp. negative) 
if $c_{1}^{B}$ (resp. $-c_{1}^{B}$) is represented by a transverse K\"{a}hler form. 
This condition is expressed by $c_{1}^{B}>0$ (resp. $c_{1}^{B}<0$). 
We say that $(g^{T},\omega^{T})$ is 
a {\it transverse K\"{a}hler-Einstein structure} 
with Einstein constant $\kappa$ if 
$(g^{T},\omega^{T})$ is the transverse K\"{a}hler structure satisfying 
${\rm Ric}^{T}=\kappa g^{T}$ which is equivalent to $\rho^{T}=\kappa \omega^{T}$. 
If $S$ admits such a structure, 
then $2\pi c_{1}^{B}=\kappa[\omega^{T}]$, 
so the basic first Chern class has to be positive, zero or negative 
according to the sign of $\kappa$. 

We define a new Sasaki structure 
fixing the Reeb vector field $\xi$ and varying $\eta$ as follows. 
We define $\widetilde{\eta}$ by 
\begin{equation*}\label{s2.3eq3}
\widetilde{\eta}=\eta+2d_{B}^{c}\phi
\end{equation*}
for a basic function $\phi$ on $(S,\mathcal{F}_{\xi})$, 
where $d_{B}^{c}=-\phi\circ d_{B}=\sqrt{-1}\,(\overline{\partial}_{B}-\partial_{B})$. 
It implies that 
\begin{equation*}\label{s2.3eq4}
d\widetilde{\eta}=d\eta+2d_{B}d_{B}^{c}\phi=d\eta+2\sqrt{-1}\,\partial_{B}\overline{\partial}_{B}\phi.
\end{equation*}
If we choose a small $\phi$ such that $\widetilde{\eta}\wedge (d\widetilde{\eta})^{n}\neq 0$, 
then $\frac{1}{2}d\widetilde{\eta}$ is a transverse K\"{a}hler form 
for the same transverse complex structure $\Phi$. 
Putting  
\begin{equation*}\label{s2.3eq5}
\tilde{r}=r\exp\phi,
\end{equation*}
then we obtain 
\begin{equation*}\label{s2.3eq6}
\tilde{r}\frac{\partial}{\partial\tilde{r}}=r\frac{\partial}{\partial r}
\end{equation*} 
on the cone $C(S)$. 
It implies that the holomorphic structure $J$ on $C(S)$ is unchanged.  
The function $\frac{1}{2}\tilde{r}^{2}$ on $C(S)$ is a new K\"{a}hler potential, 
that is $\frac{1}{2}d(\tilde{r}^{2}\widetilde{\eta})=
\frac{\sqrt{-1}}{2}dd^{c}\tilde{r}^{2}$, since 
\begin{equation*}\label{s2.3eq7}
\widetilde{\eta}=\eta+2d_{B}^{c}\phi=2d^{c}\log\tilde{r}.
\end{equation*}
Thus the deformation 
\begin{equation}\label{s2.3eq8}
\eta\rightarrow  \widetilde{\eta}=\eta+2d_{B}^{c}\phi
\end{equation}
gives a new Sasaki structure with the same Reeb vector field, 
the same transverse complex structure and the same holomorphic structure of $C(S)$. 
Conversely, any other Sasaki structure on a compact manifold $S$ with the same Reeb vector field, 
the same transverse complex structure and the same holomorphic structure of $C(S)$ 
is given by the deformation (\ref{s2.3eq8}), by using 
the transverse $\partial\overline{\partial}$-lemma proved in $\cite{EK}$. 
The deformations (\ref{s2.3eq8}) are called {\it transverse K\"{a}hler deformations}.

\subsection{Sasaki-Einstein structures and weighted Calabi-Yau structures}
In this section, we assume that $S$ is a compact manifold. 
We provide the definition of Sasaki-Einstein manifolds. 
\begin{defi}
{\rm 
A Sasaki manifold $(S,g)$ is {\it Sasaki-Einstein} 
if the metric $g$ is Einstein. 
}
\end{defi}

Let $(\xi, \eta, \Phi, g)$ be a Sasaki structure on $S$. 
Then the Ricci tensor ${\rm Ric}$ of $g$ has following relations :
\begin{eqnarray*}
&&{\rm Ric}(u,\xi)=2n\eta(u),\ u\in TS \label{s2.5eq1}\\
&&{\rm Ric}(u,v)={\rm Ric}^{T}(u,v)-2g(u,v),\ u,v \in D \label{s2.5eq2}
\end{eqnarray*}
Thus the Einstein constant of a Sasaki-Einstein metric $g$ 
has to be $2n$, that is, ${\rm Ric}=2ng$. 
It follows from the above equations that the Einstein condition ${\rm Ric}=2ng$ 
is equal to ${\rm Ric}^{T}=2(n+1)g^{T}$. 
Moreover, the cone metric $\overline{g}$ is Ricci-flat on $C(S)$ if and only if 
$g$ is Einstein with the Einstein constant $2n$ on $S$ (we refer to Lemma 11.1.5 in \cite{BG2}). 
Hence we can characterize the Sasaki-Einstein condition as follows 
\begin{prop}\label{s2.5p1}
Let $(S,g)$ be a Sasaki manifold of dimension $2n+1$. 
Then the following conditions are equivalent. 
\begin{enumerate}
\item $(S,g)$ is a Sasaki-Einstein manifold. 

\item $(C(S),\overline{g})$ is Ricci-flat, that is, ${\rm Ric}_{\overline{g}}=0$. 

\item $g^{T}$ is transverse K\"{a}hler-Einstein with ${\rm Ric}^{T}=2(n+1)g^{T}$. 
$\hfill\Box$
\end{enumerate}
\end{prop}
We remark that Sasaki-Einstein manifolds have finite fundamental groups 
from Mayer's theorem. 
From now on, we assume that $S$ is simply connected. 
Any Sasaki-Einstein manifold associates a transverse K\"{a}hler-Einstein structure 
with positive first Chern class $c_{1}^{B}=\frac{n+1}{2\pi}[d\eta]\in H_{B}^{1,1}(S)$. 
Thus $c_{1}^{B}>0$ and $c_{1}(D)=0$ are necessary conditions 
for a Sasaki metric to admit a deformation of transverse K\"{a}hler structures 
to a Sasaki-Einstein metric. The following lemma is formalized in \cite{FOW} : 
 
\begin{lem}\label{s2.5l2}
A Sasaki manifold $(S,g)$ satisfies $c_{1}^{B}>0$ and $c_{1}(D)=0$ 
if and only if there exists a holomorphic section $\Omega$ of $K_{C(S)}$ 
with $L_{r\frac{\partial}{\partial r}}\Omega=(n+1)\Omega$ and 
\begin{equation*}\label{s2.5eq3}
\Omega\wedge \overline{\Omega}=e^{h}c_{n+1}\omega^{n+1}
\end{equation*}
for a basic function $h$ on $C(S)$, 
where $\omega=\frac{1}{2}d(r^{2}\eta)$ and 
$c_{n+1}=\frac{1}{(n+1)!}(-1)^{\frac{n(n+1)}{2}}(\frac{2}{\sqrt{-1}})^{n+1}$. 
\end{lem}
\begin{proof}
If $c_{1}^{B}>0$ and $c_{1}(D)=0$, 
then it follows that $c_{1}^{B}=a[d\eta]$ for a positive constant $a$ 
from $c_{1}(D)=\iota c_{1}^{B}$ for the map $\iota$ in the long exact sequence 
\begin{equation*}\label{s2.5eq4}
\cdots\longrightarrow H_{B}^{0}(S)\xrightarrow{\ \ \delta \ \ } H_{B}^{2}(S)
\xrightarrow{\ \ \iota \ \ } H^{2}(S;\mathbb{R})\longrightarrow \cdots
\end{equation*}
where $\delta(a)=a[d\eta]$ and $\iota[\alpha]=[\alpha]$. 
By a $D$-homothetic transformation, 
we may assume that $[\rho^{T}]=(n+1)[d\eta]\in H_{B}^{1,1}(S)$. 
Let $\rho$ be the Ricci form of the K\"{a}hler metric $\overline{g}$. 
It follows from the transverse $\partial\overline{\partial}$-lemma 
that there exists a basic function $h$ on $C(S)$ such that 
$\rho=\sqrt{-1}\,\partial\overline{\partial}h$. 
%For an positive integer $l$, 
%the $l$-th power of the canonical line bundle $K_{C(S)}^{l}$ is trivial 
%since $C(S)$ have a finite fundamental group. 
Then $e^{h}\omega^{n+1}$ induces a flat metric on $K_{C(S)}$. 
Hence we can choose a nowhere vanishing holomorphic section $\Omega$ 
of $K_{C(S)}$ such that 
\begin{equation}\label{s2.5eq5}
\Omega\wedge \overline{\Omega}=e^{h}c_{n+1}\omega^{n+1}. 
\end{equation}
The Lie derivative 
$L_{r\frac{\partial}{\partial r}}\Omega$ 
is a holomorphic $(n+1)$-form 
since the vector field $v=\frac{1}{2}(r\frac{\partial}{\partial r}-\sqrt{-1}\xi)$ is holomorphic and 
$L_{r\frac{\partial}{\partial r}}\Omega=L_{v}\Omega+L_{\overline{v}}\Omega=L_{v}\Omega$. 
Hence there exists a holomorphic function $f$ such that 
\begin{equation*}
L_{r\frac{\partial}{\partial r}}\Omega=-\sqrt{-1}L_{\xi}\Omega=f\Omega.
\end{equation*}
Taking the Lie derivatives $L_{r\frac{\partial}{\partial r}}$ and $L_{\xi}$ on the equation (\ref{s2.5eq5}), 
then we have 
\begin{eqnarray*}
&&(f+\overline{f})\Omega\wedge\overline{\Omega}=2(n+1)e^{h}c_{n+1}\omega^{n+1}, \\
&&(f-\overline{f})\Omega\wedge\overline{\Omega}=0 
\end{eqnarray*}
since $L_{r\frac{\partial}{\partial r}}\omega=2\omega$ 
and $L_{\xi}\omega=0$. 
It yields that $f=n+1$, and hence we obtain 
\begin{equation*}
L_{r\frac{\partial}{\partial r}}\Omega=(n+1)\Omega.
\end{equation*}

Conversely, if we have such a holomorphic section $\Omega$ 
then the equation $\rho=\sqrt{-1}\,\partial\overline{\partial}h$ holds. 
It implies $\rho^{T}=(n+1)d\eta+\sqrt{-1}\,\partial\overline{\partial}h$. 
Hence we obtain $c_{1}^{B}>0$ and $c_{1}(D)=0$, and it completes the proof.  
\end{proof}

\begin{defi}\label{s5.2d1}
{\rm 
A pair $(\Omega,\omega)\in\wedge^{n+1}\otimes\mathbb{C}\oplus\wedge^{2}$ is called 
a \textit{weighted Calabi-Yau structure} on $C(S)$ 
if $\Omega$ is a holomorphic section of $K_{C(S)}$ 
and $\omega$ is a K\"{a}hler form satisfying 
%Monge-Amp\`{e}re% 
the equation 
\begin{equation*}\label{s2.5eq7}
\Omega\wedge \overline{\Omega}=c_{n+1}\omega^{n+1}
\end{equation*}
where 
$c_{n+1}=\frac{1}{(n+1)!}(-1)^{\frac{n(n+1)}{2}}(\frac{2}{\sqrt{-1}})^{n+1}$ 
and 
\begin{eqnarray*}
&&L_{r\frac{\partial}{\partial r}}\Omega=(n+1)\Omega, \label{s2.5eq8}\\
&&L_{r\frac{\partial}{\partial r}}\omega=2\omega. \label{s2.5eq9}
\end{eqnarray*} 
}
\end{defi}
If there exists a weighted Calabi-Yau structure $(\Omega,\omega)$ on $C(S)$, 
then it is unique up to change 
$\Omega\to e^{\sqrt{-1}\,\theta}\Omega$ of a phase $\theta\in \mathbb{R}$. 

\begin{prop}\label{s2.5p3}
A Riemannian metric $g$ on $S$ is Sasaki-Einstein 
if and only if there exists 
a weighted Calabi-Yau structure $(\Omega,\omega)$ on $C(S)$ 
such that $\overline{g}$ is the K\"{a}hler metric. 
\end{prop}
\begin{proof}
We assume that $(S,g)$ is Sasaki-Einstein. 
Then $c_{1}^{B}>0$ and $c_{1}(D)=0$ since $c_{1}^{B}=\frac{n+1}{2\pi}[d\eta]\in H_{B}^{2}(S)$. 
Hence Lemma~\ref{s2.5l2} implies that 
there exists a holomorphic section $\Omega^{\prime}$ 
of $K_{C(S)}$ satisfying $L_{r\frac{\partial}{\partial r}}\Omega^{\prime}=(n+1)\Omega^{\prime}$ 
and the equation 
\begin{equation*}\label{s2.5eq10}
\Omega^{\prime}\wedge \overline{\Omega}^{\prime}=e^{h}c_{n+1}\omega^{n+1}
\end{equation*}
for a basic function $h$ on $C(S)$, 
where $\omega=\frac{1}{2}d(r^{2}\eta)$. 
Then $h$ must be a constant $c$ since the cone $(C(S),\overline{g})$ is Ricci flat. 
We define a holomorphic section $\Omega$ of $K_{C(S)}$ 
as 
\[
\Omega=e^{-\frac{c}{2}}\Omega^{\prime}.
\] 
Then 
%$(\Omega,\omega)$ satisfies the equation 
%$\Omega^{\prime}\wedge \overline{\Omega}^{\prime}=c_{n+1}\omega^{n+1}$ 
%and $L_{r\frac{\partial}{\partial r}}\Omega=(n+1)\Omega$. 
$(\Omega,\omega)$ is a weighted Calabi-Yau structure on $C(S)$. 

Conversely, we assume that $(\Omega,\omega)$ is a weighted Calabi-Yau structure on $C(S)$ 
with the K\"{a}hler metric $\overline{g}$. 
Then $\overline{g}$ is Ricci-flat. 
%It yields that $\rho^{T}=(n+1)d\eta$ by the relation $\rho=\rho^{T}-(n+1)d\eta$. 
%Thus $g^{T}$ is a transverse K\"{a}hler-Einstein structure 
%with ${\rm Ric}^{T}=2(n+1)g^{T}$. 
Hence Proposition~\ref{s2.5p1} implies that $(S, g)$ is a Sasaki-Einstein manifold,  
and we finish the proof. 
\end{proof}

\section{Special Legendrian submanifolds}
We assume that $(S,g)$ is a smooth compact Riemannian manifold of dimension $(2n+1)$ 
which is greater than or equal to five. 
Let $(C(S),\overline{g})$ be the metric cone of $(S,g)$.

\subsection{Special Legendrian submanifolds and special Lagrangian cones}
We assume that $(S,g)$ is a simply connected Sasaki-Einstein manifold 
and fix a weighted Calabi-Yau structure $(\Omega, \omega)$ on $C(S)$ 
such that $\overline{g}$ is the K\"{a}hler metric. 
The real part $(e^{\sqrt{-1}\,\theta}\Omega)^{\rm Re}$ of $e^{\sqrt{-1}\,\theta}\Omega$ 
is a calibration 
whose calibrated submanifolds are called {\it $\theta$-special Lagrangian submanifolds}. 
We consider such submanifolds of cone type. 
For any submanifold $L$ in $S$, 
the cone $C(L)=\mathbb{R}_{>0}\times L$ is a submanifold in $C(S)$. 
We identify $L$ with the hypersurface $\{1\}\times L$ in $C(L)$. 
Then $L$ is considered as the link $C(L)\cap S$. 
\begin{defi}\label{s3.1d1}
{\rm 
A submanifold $L$ in $S$ is {\it special Legendrian} if and only if 
the cone $C(L)$ is a $\theta$-special Lagrangian submanifold in $C(S)$ 
for a phase $\theta$. 
}
\end{defi}
A $\theta$-special Lagrangian cone $C(L)$ is a minimal submanifold in $C(S)$, 
that is, the mean curvature vector field $\widetilde{H}$ of $C(L)$ vanishes. 
Then the mean curvature vector field $H$ of the link $L$ in $S$ satisfies that 
\[
\widetilde{H}_{(r,x)}=\frac{1}{r^{2}}H_{x}
\]
at $(r,x)\in \mathbb{R}_{>0}\times S=C(S)$. 
Hence any special Legendrian submanifold $L$ is also minimal. 
Conversely, we assume that $L$ is 
a connected oriented minimal Legendrian submanifold in $S$. 
Then the cone $C(L)$ is minimal Lagrangian. 
There exists a function $\theta$ on $C(L)$ such that 
$*(\Omega|_{C(L)})=e^{\sqrt{-1}\,\theta}$ where $*$ is the Hodge operator 
with respect to the metric $\overline{g}|_{L}$ on $L$ induced by $\overline{g}$. 
We have 
\[
X(\theta)=-\omega(\widetilde{H}, X)
\] 
for any vector filed $X$ on $C(S)$ tangent to $C(L)$ (Lemma 2.1. \cite{TY}). 
It yields that $\theta$ is constant. 
Thus, the cone $C(L)$ is special Lagrangian with respect to 
a weighted Calabi-Yau structure $(e^{\sqrt{-1}\,\theta}\Omega, \omega)$ 
for a phase $\theta$. 
Hence $C(L)$ is $\theta$-special Lagrangian, 
and the link $L=C(L)\cap S$ is a special Legendrian submanifold. 
We obtain the following 
(for the case of the sphere $S^{2n+1}$, we refer to Proposition 26 \cite{H}) : 
\begin{prop}\label{s3.1p1}
A connected oriented Legendrian submanifold in $S$ is minimal 
if and only if it is special Legendrian. $\hfill\Box$
\end{prop}

Let $(\xi, \eta, \Phi, g)$ be the corresponding Sasaki structure on $S$. 
We also denote by $\eta$ the extension to $C(S)$. 
We provide a characterization of special Lagrangian cones in $C(S)$. 
\begin{prop}\label{s3.1p2}
An $(n+1)$-dimensional closed submanifold $\widetilde{L}$ in $C(S)$ 
is a special Lagrangian cone if and only if 
$\Omega^{\rm Im} |_{\widetilde{L}}=0$ and $\eta |_{\widetilde{L}}=0$. 
\end{prop}
\begin{proof}
We note that a special Lagrangian submanifold in $C(S)$ is characterized by 
an $(n+1)$-dimensional submanifold $\widetilde{L}$ in $C(S)$ such that 
$\Omega^{\rm Im} |_{\widetilde{L}}=0$ and $\omega |_{\widetilde{L}}=0$. 
If $\widetilde{L}$ is a special Lagrangian cone, 
then the vector field $r\frac{\partial}{\partial r}$ is tangent to $\widetilde{L}$. 
The vector fields $\xi$ and $r\frac{\partial}{\partial r}$ span a symplectic subspace 
of $T_{p}C(S)$ with respect to $\omega_{p}$ at the each point $p\in C(S)$. 
We can obtain $\eta|_{\widetilde{L}}=0$ 
since $\eta=i_{r\frac{\partial}{\partial r}}\omega$ and $\omega |_{\widetilde{L}}=0$. 

Conversely, if an $(n+1)$-dimensional submanifold $\widetilde{L}$ satisfies 
$\Omega^{\rm Im} |_{\widetilde{L}}=0$ and $\eta |_{\widetilde{L}}=0$, 
then $\widetilde{L}$ is a special Lagrangian submanifold 
since $\omega |_{\widetilde{L}}=\frac{1}{2}d(r^{2}\eta|_{\widetilde{L}})=0$. 
In order to see that $\widetilde{L}$ is a cone, we consider the set 
\[
I_{p}=\{ a \in \mathbb{R}_{>0}\mid \lambda_{a}p\in \widetilde{L}\}
\]
for each $p\in \widetilde{L}$. 
Then $I_{p}$ is a closed subset of $\mathbb{R}_{>0}$ since $\widetilde{L}$ is closed. 
On the other hand, the vector field $r\frac{\partial}{\partial r}$ 
has to be tangent to $\widetilde{L}$ 
since $\widetilde{L}$ is Lagrangian and $\eta|_{\widetilde{L}}=0$.  
The vector field $r\frac{\partial}{\partial r}$ is the infinitesimal transformation 
of the action $\lambda$. 
Therefore $I_{p}$ is open, 
and so $I_{p}=\mathbb{R}_{>0}$ for each point $p\in \widetilde{L}$. 
Hence $\widetilde{L}$ is a cone, and it completes the proof. 
\end{proof}

Many compact special Lagrangian submanifolds are obtained as the fixed point sets of 
anti-holomorphic involutions of compact Calabi-Yau manifolds. 
Bryant constructs special Lagrangian tori in Calabi-Yau $3$-folds by the method \cite{Br}. 
We apply the method to find special Legendrian submanifolds in Sasaki-Einstein manifolds. 
An anti-holomorphic involution $\tau$ of $C(S)$ is 
a diffeomorphism $\tau:C(S)\to C(S)$ with $\tau^{2}={\rm id}$ 
and $\tau_{*}\circ J=-J\circ \tau_{*}$ 
where $J$ is the complex structure on $C(S)$ induced by the Sasaki structure. 
\begin{prop}\label{s3.1p3}
We assume there exists an anti-holomorphic involution $\tau$ of $C(S)$ 
such that $\tau^{*}r=r$. 
If the set ${\rm fix}(\tau)$ is not empty, 
then the link ${\rm fix}(\tau)\cap S$ is a special Legendrian submanifold in $S$. 
\end{prop}
\begin{proof}
Let $(\Omega, \omega)$ be a weighted Calabi-Yau structure on $C(S)$ such that 
$\omega=\frac{1}{2}d(r^{2}\eta)$. 
Then we have 
\begin{equation*}
\tau^{*}\eta=\tau^{*}\circ d^{c}\log r=
-d^{c}\circ \tau^{*}\log r=-d^{c}\log r=-\eta
\end{equation*}
since $\tau^{*}\circ d^{c}=-d^{c}\circ\tau^{*}$ and $\tau^{*}r=r$. 
It yields that $\tau^{*}\omega=-\omega$ and $\tau$ is an isometry. 
There exists a holomorphic function $f$ on $C(S)$ such that 
\begin{equation}\label{s3.1eq1}
\overline{\tau^{*}\Omega}=f\Omega. 
\end{equation}
The Lie derivative $L_{r\frac{\partial}{\partial r}}$ satisfies 
that $L_{r\frac{\partial}{\partial r}}\circ \tau^{*}=
\tau^{*}\circ L_{r\frac{\partial}{\partial r}}$ and 
$L_{r\frac{\partial}{\partial r}}\Omega =(n+1)\Omega$. 
We also have $L_{\xi}\Omega=\sqrt{-1}\,(n+1)\Omega$. 
Taking the Lie derivative $L_{r\frac{\partial}{\partial r}}$ 
on the equation (\ref{s3.1eq1}), 
then we obtain that $L_{r\frac{\partial}{\partial r}}f=0$ and $L_{\xi}f=0$. 
Thus $f$ is the pull-back of a basic and transversely holomorphic function on $S$. 
Hence $f$ is constant. 
Moreover, the equation (\ref{s3.1eq1}) implies that 
$f=e^{2\sqrt{-1}\,\theta}$ for a real constant $\theta$ since the map $\tau$ is an isometry.  
We denote by $\Omega_{\theta}$ the holomorphic $(n+1)$-form $e^{\sqrt{-1}\,\theta}\Omega$. 
Then $(\Omega_{\theta}, \omega)$ is a weighted Calabi-Yau structure on $C(S)$ such that 
\begin{equation*}
\tau^{*}\Omega_{\theta}=\overline{\Omega}_{\theta}. 
\end{equation*}
The set ${\rm fix}(\tau)$ is 
an $(n+1)$-dimensional closed submanifold, if it is not empty, 
since $\tau$ is an isometric and anti-holomorphic involution. 
We denote the manifold ${\rm fix}(\tau)$ by $\widetilde{L}$. 
Since $\tau$ is the identity map on $\widetilde{L}$, 
we have 
\begin{equation*}
\Omega_{\theta}|_{\widetilde{L}}=\tau^{*}\Omega_{\theta}|_{\widetilde{L}}=\overline{\Omega}_{\theta}|_{\widetilde{L}}. 
\end{equation*}
It yields that $\Omega_{\theta}^{\rm Im} |_{\widetilde{L}}=0$. 
Therefore, Proposition~\ref{s3.1p3} implies that 
$\widetilde{L}$ is a $\theta$-special Lagrangian cone in $C(S)$.  
Then the link $\widetilde{L}\cap S$ is a special Legendrian submanifold 
in $S$, and it completes the proof. 
\end{proof}

A real form of a K\"{a}hler manifold is 
a totally geodesic Lagrangian submanifold \cite{O}. 
We can generalize Proposition~\ref{s3.1p3} to Sasaki manifolds 
which are not necessarily Einstein and simpliy connected as follows : 
\begin{prop}\label{s3.1p4}
Let $(S,g)$ be a Sasaki manifold. 
We assume there exists an anti-holomorphic involution $\tau$ of $C(S)$ 
such that $\tau^{*}r=r$. 
If the set ${\rm fix}(\tau)$ is not empty, 
then the link ${\rm fix}(\tau)\cap S$ is a totally geodesic Legendrian submanifold in $S$. 
\end{prop}
\begin{proof}
We remark that $\tau$ satisfies $\tau^{*}\eta=-\eta$ and $\tau^{*}\omega=-\omega$. 
The fixed point set ${\rm fix}(\tau)$ of the anti-symplectic involution $\tau$ 
is a Lagrangian submanifold in $C(S)$ if it is not empty. 
Moreover, any closed Lagrangian submanifold where $\eta$ vanishes is a cone 
as in the proof of Proposition~\ref{s3.1p3}. 
Since $\eta|_{\widetilde{L}}=0$ holds, 
the set ${\rm fix}(\tau)$ is a Lagrangian cone in $C(S)$ 
and induces a Legendrian submanifold ${\rm fix}(\tau)\cap S$ as the link. 
The restriction $\tau|_{S}$ of $\tau$ to $S$ induces a map from $S$ to itself 
since $\tau$ preserves a level set of $r$. 
Then ${\rm fix}(\tau)\cap S$ is the fixed point set ${\rm fix}(\tau|_{S})$ of $\tau|_{S}$. 
The set ${\rm fix}(\tau|_{S})$ is totally geodesic 
since the map $\tau|_{S}$ is an isometric involution on $(S,g)$. 
Hence ${\rm fix}(\tau)\cap S$ is a totally geodesic Legendrian submanifold. 
\end{proof}

\begin{rem}\label{s3.1r1}
{\rm 
Tomassini and Vezzoni introduced 
a special Legendrian submanifold 
in a contact Calabi-Yau manifold which is a contact manifold 
with a transversely Calabi-Yau foliation~\cite{TV}. 
A contact Calabi-Yau manifold is a Sasaki manifold 
with a transversely null K\"{a}hler-Einstein structure. 
Hence it is not Sasaki-Einstein. 
} 
\end{rem}

%\section{Special Legendrian submanifolds in toric Sasaki-Einstein manifolds}
%In this section, we will show that 
%there exist a special Legendrian submanifolds in any toric Sasaki-Einstein manifold. 

\subsection{Toric Sasaki manifolds}
In this section, we consider the toric Sasaki manifolds. 
We refer to~\cite{CFO}, \cite{Gu} and \cite{L} 
for some facts of toric Sasaki manifolds.
We provide the definition of toric Sasaki manifolds. 
\begin{defi}\label{s3.2d1}
{\rm 
A Sasaki manifold $(S,g)$ is \textit{toric} if 
there exists an effective action of an $(n+1)$-torus $\mathbb{T}^{n+1}=G$ 
preserving the Sasaki structure such that 
the Reeb vector field $\xi$ is an element of the Lie algebra $\mathfrak{g}$ of $G$. 
Equivalently, a toric Sasaki manifold $(S,g)$ is a Sasaki manifold 
whose metric cone $(C(S),\overline{g})$ is a toric K\"{a}hler cone. 
} 
\end{defi}

We define the moment map 
\begin{equation*}\label{s3.2eq1}
\tilde{\mu}:C(S)\to \mathfrak{g}^{*} 
\end{equation*}
of the action $G$ on $C(S)$ by 
\begin{equation}\label{s3.2eq2}
\langle\tilde{\mu},\zeta\rangle=\frac{1}{2}r^{2}\eta(X_{\zeta})  
\end{equation}
for any $\zeta \in \mathfrak{g}$, 
where $X_{\zeta}$ is the vector field on $C(S)$ 
induced by $\zeta \in \mathfrak{g}$. 
Let $G_{\mathbb{C}}=(\mathbb{C}^{*})^{n+1}$ denote the complexification of $G$. 
The action $G_{\mathbb{C}}$ on the cone $C(S)$ is holomorphic and has an open dense orbit. 
The restriction of $\tilde{\mu}$ to $S$ is a moment map 
of the action $G$ on $S$. 
The equation (\ref{s3.2eq2}) implies that 
$\tilde{\mu}(S)=\{y\in \mathfrak{g}^{*} \mid 
\langle y,\xi\rangle = \frac{1}{2} \}$. 
The hyperplane $\{y\in \mathfrak{g}^{*} \mid 
\langle y,\xi\rangle = \frac{1}{2} \}$ is called 
the {\it characteristic hyperplane} \cite{BG1}. 
We define $C(\tilde{\mu})$ by 
\begin{equation*}\label{s3.2eq4}
C(\tilde{\mu})=\tilde{\mu}(C(S))\cup\{0\}. 
\end{equation*} 
Then we obtain 
\begin{equation*}\label{s3.2eq5}
C(\tilde{\mu})=\{t\xi\in\mathfrak{g}^{*} \mid \xi\in \tilde{\mu}(S),\ t\in[0,\infty) \}. 
\end{equation*}
The cone $C(\tilde{\mu})$ is called the {\it moment cone} 
of the toric Sasaki manifold. 
 
We provide the definition of a {\it good rational polyhedral cone} which is due to Lerman~\cite{L} ; 
\begin{defi}\label{s3.2d2}
{\rm 
Let $\mathbb{Z}_{\mathfrak{g}}$ be the integral lattice of $\mathfrak{g}$, 
which is the kernel of the exponential map ${\rm exp}:\mathfrak{g}\to G$. 
A subset $C$ of $\mathfrak{g}^{*}$ is a {\it rational polyhedral cone} if 
there exist an integer $d\ge n+1$ and vectors $\lambda_{i}\in \mathbb{Z}_{\mathfrak{g}}$, 
$i=1,\dots,d$, such that  
\begin{equation*}\label{s3.2eq6}
C=\{y \in \mathfrak{g}^{*} \mid \langle y,\lambda_{i}\rangle\ge 0 \ {\rm for}\ i=1,\dots, d \}. 
\end{equation*}
The set $\{\lambda_{i}\}$ is {\it minimal} if  
\begin{equation*}\label{s3.2eq7}
C\neq \{y \in \mathfrak{g}^{*} \mid \langle y,\lambda_{i}\rangle\ge 0 \ {\rm for}\ i\neq j \} 
\end{equation*}
for any $j$, and is {\it primitive} 
if there does not exist an integer $n_{i}(\ge 2)$ 
and $\lambda_{i}^{\prime}\in \mathbb{Z}_{\mathfrak{g}}$ such that 
$\lambda_{i}=n_{i}\lambda_{i}^{\prime}$ for each $i$. 
A rational polyhedral cone $C$ such that $\{\lambda_{i}\}$ is minimal 
and primitive is called {\it good} 
if $C$ has non-empty interior and satisfies the following condition : 
if 
\begin{equation*}\label{s3.2eq8}
\{y \in C\ \mid \ \langle y,\lambda_{i_{j}}\rangle =0\ {\rm for}\ j=1,\dots, k \} 
\end{equation*}
is non-empty face of $C$ for some $\{i_{1},\dots, i_{k}\}\subset \{1,\dots,d\}$, 
then $\{\lambda_{i_{1}},\dots, \lambda_{i_{k}}\}$ is linearly independent over $\mathbb{Z}$ 
and 
\begin{equation}\label{s3.2eq9}
\left\{ \sum_{j=1}^{k}a_{j}\lambda_{i_{j}}\ \Big|\  a_{j}\in \mathbb{R} \right\} \cap \mathbb{Z}_{\mathfrak{g}}=
\left\{\sum_{j=1}^{k}m_{j}\lambda_{i_{j}}\ \Big|\ m_{j}\in \mathbb{Z} \right\}
\end{equation}
}
\end{defi}
Any moment cone of toric Sasaki manifolds of ${\rm dim}\ge 5$ 
is a good rational polyhedral cone which is strongly convex, 
that is, the cone does not contain non-zero linear subspace (cf. Proposition 4.38. \cite{BG2}). 
Conversely, Given a strongly convex good rational polyhedral cone 
we can obtain a toric Sasaki manifold by Delzant construction. 

\begin{prop}\label{s3.2p1}
If $C$ is a strongly convex good rational polyhedral cone and $\xi$ is an element of 
\begin{equation*}\label{s3.2eq10}
C^{*}_{0}=\{\xi \in \mathfrak{g}\ \mid \ \langle v,\xi\rangle >0,\ ^{\forall} v\in C \}, 
\end{equation*}
then there exists a connected toric Sasaki manifold $S$ 
with the Reeb vector field $\xi$ such that the moment cone is $C$. 
\end{prop}

\noindent
\textit{Outline of the proof.} 
Let $\{e_{1},\dots,e_{d}\}$ be the canonical basis of $\mathbb{R}^{d}$. 
The basis generates the lattice $\mathbb{Z}^{d}$. 
Let $\beta:\mathbb{R}^{d}\to\mathfrak{g}$ be the linear map defined by 
\begin{equation*}\label{s3.2eq11}
\beta(e_{i})=\lambda_{i} 
\end{equation*}
for $i=1,\dots, d$. 
Since the polyhedral cone $C$ has non-empty interior, 
there exists a basis $\{\lambda_{i_{1}},\cdots,\lambda_{i_{n+1}} \}$ of $\mathfrak{g}$ over $\mathbb{R}$. 
Thus the map $\beta$ is surjective. 
The map $\beta$ induces the map $\widetilde{\beta}$ 
from $\mathbb{T}^{d}\cong \mathbb{R}^{d}/\mathbb{Z}^{d}$ 
to $G\cong \mathfrak{g}/\mathbb{Z}_{\mathfrak{g}}$. 
Let $K$ denote the kernel of $\widetilde{\beta}$. 
Then we have  
\begin{equation*}\label{s3.2eq12}
0\to K\xrightarrow{\widetilde{\iota}} \mathbb{T}^{d} \xrightarrow{\widetilde{\beta}} G \to 0 
\end{equation*}
where $\widetilde{\iota}$ is the natural monomorphism. 
The group $K$ is a compact abelian subgroup of $\mathbb{T}^{d}$ and represented by 
%\begin{equation*}\label{s3.2eq13}
$K=
\left\{ [a]\in \mathbb{T}^{d} \mid 
\sum_{i=1}^{d}a_{i}\lambda_{i} \in \mathbb{Z}_{\mathfrak{g}} \right\}$ 
%\end{equation*}
where $[a]$ denotes the equivalent class of $a\in\mathbb{R}^{d}$. 
Let $\mathfrak{k}$ denote the Lie algebra of $K$. 
Then $\mathfrak{k}$ is equal to $\ker\beta$. 
Thus we obtain the exact sequence 
\begin{equation}\label{s3.2eq14}
0\to \mathfrak{k}\xrightarrow{\iota}\mathbb{R}^{d} \xrightarrow{\beta} \mathfrak{g}\to 0 
\end{equation}
where $\iota$ is the natural inclusion. 
The action of $\mathbb{T}^{d}$ on $\mathbb{C}^{d}$ is given by 
\begin{equation*}\label{s3.2eq15}
[a]\circ(z_{1},\dots,z_{d})=(e^{2\pi\sqrt{-1}\,a_{1}}z_{1},\dots,e^{2\pi\sqrt{-1}\,a_{d}}z_{d})
\end{equation*}
for $[a]=[a_{1},\dots,a_{d}]
\in \mathbb{T}^{d}\cong \mathbb{R}^{d}/\mathbb{Z}^{d}$ 
and $(z_{1},\dots,z_{d})\in\mathbb{C}^{d}$. 
This action preserves the standard K\"{a}hler form on $\mathbb{C}^{d}$. 
The corresponding moment map 
\begin{equation*}\label{s3.2eq16}
\mu_{0}:\mathbb{C}^{d}\to (\mathbb{R}^{d})^{*}
\end{equation*}
is given by 
\begin{equation*}\label{s3.2eq17}
\mu_{0}(z)=\sum_{j=1}^{d}|z_{j}|^{2}e_{j}^{*}
\end{equation*}
for $z\in \mathbb{C}^{d}$ 
where $\{e_{1}^{*},\dots,e_{d}^{*}\}$ is the dual basis to $\{e_{1},\dots,e_{d}\}$. 
We choose a basis $\{v_{1},\dots,v_{k}\}$ of $\mathfrak{k}$ 
where $k=\dim\mathfrak{k}=d-n-1$, 
then there exists an integer $k\times d$-matrix $(a_{ij})$ 
such that $\iota(v_{i})=\sum_{j=1}^{d} a_{ij}e_{j}$ for $i=1,\dots,k$. 
We also consider the following exact sequence  
\begin{equation}\label{s3.2eq18}
0\to \mathfrak{g}^{*}\xrightarrow{\beta^{*}}(\mathbb{R}^{d})^{*}\xrightarrow{\iota^{*}} \mathfrak{k}^{*} \to 0 
\end{equation}
which is the dual sequence to (\ref{s3.2eq14}). 
We define a map 
\begin{equation*}\label{s3.2eq19}
\mu:\mathbb{C}^{d}\to \mathfrak{k}^{*}
\end{equation*}
by $\mu=\iota^{*}\circ \mu_{0}$, 
then $\mu$ is a moment map of the action of $K$ on $\mathbb{C}^{d}$ 
and given by 
\begin{equation*}\label{s3.2eq20}
\mu(z)=\sum_{i=1}^{k}(\sum_{j=1}^{d}a_{ij}|z_{j}|^{2})v_{i}^{*}
\end{equation*}
for $z\in \mathbb{C}^{d}$ where $\{v_{1}^{*},\dots,v_{k}^{*}\}$ is the dual basis to $\{v_{1},\dots,v_{k}\}$. 
%The K\"{a}hler quotient $\mathbb{C}^{d}//K=(\mu^{-1}(0)\backslash \{0\})/K$ 
%has an $(n+1)$-torus $G=\mathbb{T}^{d}/K$ action. 
It follows from the exact sequence (\ref{s3.2eq18}) that 
$\mu_{0}(\mu^{-1}(0))\subset \beta^{*}\mathfrak{g}^{*}\simeq \mathfrak{g}^{*}$. 
Hence we have the map $\mu_{0}|_{\mu^{-1}(0)}:\mu^{-1}(0)\to \mathfrak{g}^{*}$. 
Moreover, it induces a map $\widetilde{\mu}$ 
from the quotient space $(\mu^{-1}(0)\backslash \{0\})/K$ to $\mathfrak{g}^{*}$ : 
\begin{equation}\label{s3.2eq21}
\widetilde{\mu}:(\mu^{-1}(0)\backslash \{0\})/K\to \mathfrak{g}^{*}.
\end{equation}
Then the map $\widetilde{\mu}$ is a moment map of 
the action $G=\mathbb{T}^{d}/K$ on $(\mu^{-1}(0)\backslash \{0\})/K$. 
The image of $\widetilde{\mu}$ is equal to $C$ 
since the image $\mu_{0}(\mu^{-1}(0))$ is precisely $\beta^{*}(C)\simeq C$. 

We define $\xi_{0}$ by the element 
\begin{equation}\label{s3.2eq22}
\xi_{0}=\sum_{i=1}^{n+1}\lambda_{i}
\end{equation}
of $\mathfrak{g}$. 
We provide a K\"{a}hler metric on $(\mu^{-1}(0)\backslash \{0\})/K$ by the K\"{a}hler reduction. 
Then the function 
\begin{equation*}\label{s3.2eq23}
F_{0}(z)=\langle \widetilde{\mu}(z), \xi_{0} \rangle
\end{equation*}
is a K\"{a}hler potential on $(\mu^{-1}(0)\backslash \{0\})/K$. 
We define $r_{0}$ by the function 
\begin{equation*}\label{s3.2eq24}
r_{0}=\sqrt{2F_{0}}  
\end{equation*}
on $(\mu^{-1}(0)\backslash \{0\})/K$. 
It yield that the K\"{a}hler potential is 
$\frac{1}{2}r_{0}^{2}=F_{0}$ and 
the manifold 
\begin{equation*}\label{s3.2eq25}
S=(\mu^{-1}(0)\cap S^{2d-1})/K  
\end{equation*}
is the hypersurface $\{r_{0}=1\}$ in $(\mu^{-1}(0)\backslash \{0\})/K$. 
Then the cone $C(S)$ of $S$ is obtained as $(\mu^{-1}(0)\backslash \{0\})/K$ : 
\begin{equation*}\label{s3.2eq26}
C(S)=(\mu^{-1}(0)\backslash \{0\})/K.
\end{equation*}
The manifold $S$ admits a Sasaki structure with the Reeb vector filed $\xi_{0}$ 
such that the following embedding from $S$ into $C(S)$ is isometric : 
\begin{equation*}\label{s3.2eq27}
S= \{r_{0}=1\}\subset C(S). 
\end{equation*}
Given an element $\xi \in C^{*}_{0}$, we can obtain 
a K\"{a}hler potential $F_{\xi}$ defined by 
\begin{equation*}\label{s3.2eq28}
F_{\xi}(z)=\langle \widetilde{\mu}(z), \xi \rangle
\end{equation*}
for $z\in C(S)$ (see (61) in \cite{FOW}). 
We denote by $H_{\xi}$ the hypersurface 
\begin{equation*}\label{s3.2eq29}
H_{\xi}=\mu^{-1}_{0}(\{ y\in \mathfrak{g}^{*} \mid \langle y, \xi \rangle =\frac{1}{2} \})
\end{equation*}
in $\mathbb{C}^{d}$, which is the inverse image 
of the characteristic hyperplane by $\mu_{0}$. 
We define a non-negative function $r$ on $C(S)$ by 
\begin{equation*}\label{s3.2eq30}
r=\sqrt{2F_{\xi}}.
\end{equation*}
Then the manifold  
\begin{equation*}\label{s3.2eq31}
S_{\xi}=(\mu^{-1}(0)\cap H_{\xi})/K
\end{equation*}
is the hypersurface $\{ r=1 \}$ in $C(S)$. 
We remark that $S_{\xi}$ is also the inverse image of 
the characteristic hyperplane by $\widetilde{\mu}$ : 
\begin{equation*}\label{s3.2eq32}
S_{\xi}=\widetilde{\mu}^{-1}
(\{ y\in \mathfrak{g}^{*} \mid \langle y, \xi \rangle =\frac{1}{2} \}). 
\end{equation*} 
Thus it follows from $S=S_{\xi_{0}}$ that there exists a diffeomorphism 
\begin{equation*}\label{s3.2eq33}
S\simeq S_{\xi}.
\end{equation*}
The manifold $S$ can be embedded in $C(S)$ as the hypersurface $S_{\xi}=\{ r=1 \}$ :
\begin{equation}\label{s3.2eq34}
S\simeq \{r=1\}\subset C(S). 
\end{equation}
Then $S$ is a Sasaki manifold such that $\xi$ is the Reeb vector field 
and the embedding (\ref{s3.2eq34}) is isometric. $\hfill\Box$

Toric Sasaki manifolds are constructed by 
a strongly convex good rational polyhedral cone $C$ 
and a Reeb vector field $\xi\in C_{0}^{*}$. 
Then the K\"{a}hler potential can be taken by $F_{\xi}$ 
as in the proof of Proposition~\ref{s3.2p1}. 
Any toric Sasaki structure with the same Reeb vector field $\xi$ 
and the same holomorphic structure on $C(S)$ is given by 
deformations of transverse K\"{a}hler structures (See Section 2.3). 
Martelli, Sparks and Yau proved the following in \cite{MSY} :
\begin{lem}
The moduli space of toric K\"{a}hler cone metrics on $C(S)$ is 
\begin{equation*}\label{s3.2eq35}
C_{0}^{*}\times \mathcal{H}^{1}(C)
\end{equation*}
where $\xi\in C_{0}^{*}$ is the Reeb vector field and 
$\mathcal{H}^{1}(C)$ denotes the space of homogeneous degree one functions on $C$ 
such that each element $\phi$ is smooth up to the boundary and 
$\sqrt{-1}\,\partial\overline{\partial}(F_{\xi}\exp{2\widetilde{\mu}^{*}\phi})$ 
is positive definite on $C(S)$. $\hfill\Box$
\end{lem}
We identify an element $\phi$ of $\mathcal{H}^{1}(C)$ with 
the pull-back $\widetilde{\mu}^{*}\phi$ by 
the moment map $\widetilde{\mu}$ as in (\ref{s3.2eq21}). 
Then, for any element $(\xi,\phi)\in C_{0}^{*}\times \mathcal{H}^{1}(C)$ 
we can define the function $r$ on $C(S)$ by 
\begin{equation*}\label{s3.2eq36}
r=\sqrt{2F_{\xi}}\exp{\phi}.
\end{equation*}
Let $S_{\xi,\phi}$ denote the hypersurface $\{r=1\}$ in $C(S)$ : 
\begin{equation*}\label{s3.2eq37}
S_{\xi,\phi}=\{r=1\}.
\end{equation*}
It is easy to see that $S=S_{\xi_{0},0}$ where $\xi_{0}$ is given by (\ref{s3.2eq22}). 
There exists a diffeomorphism 
\begin{equation*}\label{s3.2eq38}
S\simeq S_{\xi,\phi} 
\end{equation*} 
for any $(\xi,\phi)\in C_{0}^{*}\times \mathcal{H}^{1}(C)$. 
%since $(\xi,\phi)$ is path connected to $(\xi_{0},0)$. 
Hence $S$ can be embedded in $C(S)$ as $S_{\xi,\phi}$ :
\begin{equation*}\label{s3.2eq39}
S\simeq \{r=1\}\subset C(S). 
\end{equation*} 
Then $S$ admits a Sasaki structure with the Reeb vector field $\xi$ 
and the K\"{a}hler potential $\frac{1}{2}r^{2}$ on $C(S)$. 
Thus the deformation 
\begin{equation}\label{s3.2eq40}
(\xi,\phi)\to (\xi^{\prime},\phi^{\prime})
\end{equation}
of $C_{0}^{*}\times \mathcal{H}^{1}(C)$ induces 
a deformation of Sasaki structure on $S$. 
These deformations are called {\it deformations of toric Sasaki structures} on $S$. 

\subsection{Main theorems}
Let $(S,g)$ be a toric Sasaki manifold. 
The metric cone $C(S)$ is given by the K\"{a}hler quotient 
$C(S)=(\mu^{-1}(0)\backslash \{0\})/K$ for 
the moment map $\mu:\mathbb{C}^{d}\to \mathfrak{k}^{*}$ 
as in the proof of Proposition~\ref{s3.2p1}. 
Then there exists an anti-holomorphic involution $\tau$ on $C(S)$ as follows. 
We consider the anti-holomorphic involution 
$\widetilde{\tau}:\mathbb{C}^{d}\to \mathbb{C}^{d}$ defined by 
\begin{equation*}
\widetilde{\tau}(z)=\overline{z}
\end{equation*}
for $z\in\mathbb{C}^{d}$. 
The inverse image $\mu^{-1}(0)$ is invariant under the map $\widetilde{\tau}$. 
Thus $\widetilde{\tau}$ induces a diffeomorphism of $\mu^{-1}(0)$. 
Moreover, $\widetilde{\tau}$ maps a $K$-orbit to another $K$-orbit. 
Hence we can define a map $\tau:C(S)\to C(S)$ by 
\begin{equation*}
\tau[z]=[\widetilde{\tau}(z)]=[\overline{z}]
\end{equation*}
for $[z]\in (\mu^{-1}(0)\backslash \{0\})/K=C(S)$. 
Then $\tau$ is an anti-holomorphic involution of $C(S)$. 
We recall that 
the group $G_{\mathbb{C}}$ acts holomorphically on the cone $C(S)$ 
with an open dense orbit. 
We denote by $X_{0}$ the open dense orbit of $G_{\mathbb{C}}$. 
Since the orbit $X_{0}$ is identified with $(\mathbb{C}^{*})^{n+1}$, 
we can give a coordinate $w=(w_{1},\dots,w_{n+1})$ on $X_{0}$ as 
$u_{i}=e^{w_{i}}$ for any 
$u=(u_{1},\dots,u_{n+1})\in X_{0}\subset (\mathbb{C}^{*})^{n+1}$. 
%We define $\mathcal{U}_{G_{\mathbb{C}}}$ by 
%the coordinate $(X_{0},w)$ of $C(S)$. 
Then the map $\tau$ is given by 
\[
\tau(w)=\overline{w}
\]
on the coordinate $(X_{0},w)$ on $C(S)$. 
Hence the set ${\rm fix}(\tau)$ is non-empty. 

\begin{thm}\label{s3.3t1}
Let $(S,g)$ be a compact simply connected toric Sasaki-Einstein manifold. 
Then the link ${\rm fix}(\tau)\cap S$ is a special Legendrian submanifold. 
\end{thm}
\begin{proof} 
%It follows from Lemma \ref{s3.3l1} that 
%the fixed point set ${\rm fix}(\tau)$ is non-empty 
%and an $(n+1)$-dimensional manifold. 
Let $S$ be a toric Sasaki-Einstein manifold with the Sasaki structure induced by 
the element $(\xi,\phi)\in C_{0}^{*}\times \mathcal{H}^{1}(C)$. 
Then the K\"{a}hler potential on $C(S)$ is 
\begin{equation*}
\frac{1}{2}r^{2}=F_{\xi}\exp{2\phi}. 
\end{equation*} 
It follows from $\tau^{*}\tilde{\mu}=\tilde{\mu}$ 
that $F_{\xi}$ and $\phi$ are also $\tau$-invariant. 
It gives rise to 
\[
\tau^{*}r=r.
\]
%So the link ${\rm fix}(\tau)\cap S$ is a Legendrian submanifold in $S$ 
%for any $(\xi,\phi)\in C_{0}^{*}\times \mathcal{H}^{1}(C)$. 
%We assume that there exists a deformation 
%$(\xi,\phi)\in C_{0}^{*}\times \mathcal{H}^{1}(C)$ 
%of toric Sasaki structures such that $S$ is a toric Sasaki-Einstein manifold. 
Hence, Proposition~\ref{s3.1p3} implies that 
the link ${\rm fix}(\tau)\cap S$ is a special Legendrian submanifold in $S$. 
It completes the proof. 
\end{proof}
\begin{rem}
{\rm 
In the case that $S$ is not simply connected, 
the canonical line bundle $K_{C(S)}$ is not necessarily trivial. 
However, the $l$-th power $K_{C(S)}^{l}$ of $K_{C(S)}$ is 
trivial for some integer $l$. 
Hence, we can remove the condition that $S$ is simply connected 
in Theorem \ref{s3.3t1} by considering nowhere vanishing 
holomorphic sections of $K_{C(S)}^{l}$ instead of $K_{C(S)}$. 
Then we need to define a special Lagrangian submanifold in $C(S)$ 
as a Lagrangian submanifolds whose $l$-th covering is a special Lagrangian submanifold 
in the $l$-th covering of $C(S)$. 
}
\end{rem}

In the proof of Theorem~\ref{s3.3t1}, 
we only need the Einstein condition of $(S,g)$ to use Proposition~\ref{s3.1p3}. 
By applying Proposition~\ref{s3.1p4} in stead of Proposition~\ref{s3.1p3}, 
we can prove the following : 
\begin{thm}\label{s3.3t2}
Let $S$ be a compact toric Sasaki manifold. 
Then the link ${\rm fix}(\tau)\cap S$ is 
a totally geodesic Legendrian submanifold in $S$. $\hfill\Box$
\end{thm}

\subsection{Covering spaces over the link ${\rm fix}(\tau)\cap S$}
In this section, we will see that the special Legendrian submanifold in Theorem~\ref{s3.3t1} 
is given by a base space of a finite covering map (we also refer to \cite{Gu}). 

We recall the exact sequence 
\begin{equation*}\label{s3.4eq1}
0\to K\xrightarrow{\widetilde{\iota}} \mathbb{T}^{d} 
\xrightarrow{\widetilde{\beta}} \mathbb{T}^{n+1} \to 0 
\end{equation*} 
is associated with a strongly convex good rational polyhedral cone $C$ as in Section~3.2. 
This sequence equips the following sequence 
\begin{equation*}\label{s3.4eq2}
0\to \mathfrak{k}\xrightarrow{\iota}\mathbb{R}^{d} \xrightarrow{\beta} \mathbb{R}^{n+1}\to 0. 
\end{equation*}
We consider each element $\lambda_{i}$ of 
the set $\{\lambda_{1},\dots, \lambda_{d}\}$ as a vector of $\mathbb{R}^{n+1}$. 
Then the map $\beta$ is represented by 
\begin{equation*}\label{s3.4eq3}
(\lambda_{1} \cdots \lambda_{d}):\mathbb{R}^{d}\to\mathbb{R}^{n+1}. 
\end{equation*}
where $(\lambda_{1} \cdots \lambda_{d})$ is the integer $(n+1)\times d$ matrix. 
By choosing a basis of $\mathfrak{k}$, 
the map $\iota$ is represented by the $d\times k$ matrix
\begin{equation*}\label{s3.4eq4}
A={}^{t}\! (a_{ij}):\mathbb{R}^{k}\to\mathbb{R}^{d}
\end{equation*}
where each component $a_{ij}$ is an integer and ${}^{t}\! B$ means the transpose of a matric $B$. 

In order to analyse ${\rm fix}(\tau)\cap S$, 
we define a map 
\[
\mu_{\mathbb{R}}:\mathbb{R}^{d}\to\mathfrak{k}^{*}
\]
by the restriction of the moment map 
$\mu:\mathbb{C}^{d}\to\mathfrak{k}^{*}$ to 
$\mathbb{R}^{d}={\rm fix}(\widetilde{\tau})\cap\mathbb{C}^{d}$. 
Then the map $\mu_{\mathbb{R}}$ is represented by  
\begin{equation*}\label{s3.4eq5}
\mu_{\mathbb{R}}(x)=\sum_{i=1}^{k}(\sum_{j=1}^{d}a_{ij}x_{j}^{2})v_{i}^{*}
\end{equation*}
for $x\in \mathbb{R}^{d}$ since 
$\mu(z)=\sum_{i}(\sum_{j=1}^{d}a_{ij}|z_{j}|^{2})v_{i}^{*}$ for $z\in \mathbb{C}^{d}$. 
The inverse image $\mu_{\mathbb{R}}^{-1}(0)$ is precisely 
${\rm fix}(\widetilde{\tau})\cap \mu^{-1}(0)$ : 
\begin{equation*}\label{s3.4eq6}
\mu_{\mathbb{R}}^{-1}(0)={\rm fix}(\widetilde{\tau})\cap \mu^{-1}(0).
\end{equation*}
The set ${\rm fix}(\tau)$ is the image of 
${\rm fix}(\widetilde{\tau})\cap \mu^{-1}(0)$ by 
the quotient map 
\begin{equation}\label{s3.4eq7}
\pi^{\prime}:\mu^{-1}(0)\backslash\{0\}\to (\mu^{-1}(0)\backslash\{0\})/K.
\end{equation}
Hence we have the $2^{k}$-fold map 
\begin{equation*}\label{s3.4eq8}
\pi^{\prime}:\mu_{\mathbb{R}}^{-1}(0)\backslash\{0\}\to {\rm fix}(\tau)
\end{equation*}
with the deck transformation $\{a\in K \mid a^{2}=1\}$. 
We also consider the quotient map 
\begin{equation*}\label{s3.4eq9}
\pi:\mu^{-1}(0)\cap H_{\xi}\to (\mu^{-1}(0)\cap H_{\xi})/K=S_{\xi}.
\end{equation*}
which is the restriction of (\ref{s3.4eq7}) to $\mu^{-1}(0)\cap H_{\xi}$. 
Then ${\rm fix}(\tau)\cap S_{\xi}$ is the base space of the $2^{k}$-fold map 
\begin{equation*}\label{s3.4eq10}
\pi:\mu_{\mathbb{R}}^{-1}(0)\cap H_{\xi}\to {\rm fix}(\tau)\cap S_{\xi}. 
\end{equation*}
%The identification $S_{\xi}\simeq S$ induces the diffeomorphism 
%\begin{equation*}\label{s3.2eq34}
%{\rm fix}(\tau)\cap S_{\xi}\simeq {\rm fix}(\tau)\cap S. 
%\end{equation*}
If we take an element $\xi=\sum_{j=1}^{d}b_{j}\lambda_{j}$ of $C_{0}^{*}$, 
%\begin{equation*}\label{s3.4eq12}
%\xi=\sum_{i=1}^{d}b_{i}\lambda_{i}, 
%\end{equation*}
then ${\rm fix}(\tau)\cap S_{\xi}$ is the quotient space of 
\begin{equation*}\label{s3.4eq13}
\mu_{\mathbb{R}}^{-1}(0)\cap H_{\xi}=
\left\{
x\in \mathbb{R}^{d} \ \biggm| 
\begin{array}{l}
\sum_{j=1}^{d}a_{ij}x_{j}^{2}=0,\ j=1,\dots,k \\
\sum_{j=1}^{d}b_{j}x_{j}^{2}=1
\end{array}
\right\} 
\end{equation*}
by the action of the deck transformation. 
%We remark that the space (\ref{s3.4eq13}) is independent of 
%the choice of the representation (\ref{s3.4eq12}). 

\subsection{The Sasaki-Einstein manifold $Y_{p,q}$}
In this section, 
we provide an example of special Legendrian submanifolds in $Y_{p,q}$. 
Gauntlett, Martelli, Sparks and Waldram provided 
an explicit toric Sasaki-Einstein metric $g_{p,q}$ on $S^{2}\times S^{3}$ \cite{GMSW}. 
For relatively prime non-negative integers $p$ and $q$ with $p>q$, 
the inward pointing normals to the polyhedral cone $C$ can be taken 
to be 
\begin{equation*}
\lambda_{1}={}^{t}(1,0,0),\ \lambda_{2}={}^{t}(1,p-q-1,p-q),\ 
\lambda_{3}={}^{t}(1,p,p),\ \lambda_{4}={}^{t}(1,1,0).
\end{equation*}
Then we obtain the representation matrix $A={}^{t}\! (a_{ij})$ as 
\begin{equation*}
A={}^{t}(-p-q,\ p,\ -p+q,\ p).
\end{equation*}
By the calculation in \cite{MSY}, the Reeb vector field $\xi_{min}$ 
of the toric Sasaki-Einstein metric is given by 
\begin{equation*}
\xi_{min} =(3,\ \frac{1}{2}(3p-3q+l^{-1}),\ \frac{1}{2}(3p-3q+l^{-1}))
\end{equation*}
where $l^{-1}=\frac{1}{q}(3q^{2}-2p^{2}+p\sqrt{4p^{2}-3q^{2}})$. 
Thus we can obtain 
\begin{equation*}
\mu_{\mathbb{R}}^{-1}(0)\cap H_{\xi_{min}}=
\left\{
x\in \mathbb{R}^{4} \ \biggm| 
\begin{array}{l}
px_{2}^{2}+px_{4}^{2}=(p+q)x_{1}^{2}+(p-q)x_{3}^{2} \\
(3p+3q-l^{-1})x_{1}^{2}+(3p-3q+l^{-1})x_{3}^{2}=2p
\end{array}
\right\}, 
\end{equation*}
which is diffeomorphic to $S^{1}\times S^{1}$. 
The deck transformations induces an action on $S^{1}\times S^{1}$ in $\mathbb{R}^{4}$ 
given by  
\begin{equation*}
\begin{cases}
\{{\rm id}\times{\rm id}\times{\rm id}\times{\rm id},\  
{\rm id}\times{\rm id}\times{\rm id}\times(-{\rm id})\}, & \text{$p$ : even}, \\
\{{\rm id}\times{\rm id}\times{\rm id}\times{\rm id},\  
{\rm id}\times{\rm id}\times(-{\rm id})\times{\rm id}\}, & \text{$p$ : odd, $q$ : odd}, \\
\{{\rm id}\times{\rm id}\times{\rm id}\times{\rm id},\  
{\rm id}\times{\rm id}\times(-{\rm id})\times(-{\rm id})\}, & \text{$p$ : odd, $q$ : even}. 
\end{cases}
\end{equation*}
Then the quotient space of $S^{1}\times S^{1}$ by the action is also $S^{1}\times S^{1}$ for each $(p,q)$. 
Therefore the link ${\rm fix}(\tau)\cap Y_{p,q}$ 
is also diffeomorphic to $S^{1}\times S^{1}$. 
Hence we have 
\begin{thm}
There exists a special Legendrian torus $S^{1}\times S^{1}$ in $Y_{p,q}$. $\hfill\Box$
\end{thm}

\vspace{\baselineskip}
\noindent
\textbf{Acknowledgements}. 
The author would like to thank Professor K. Fukaya and Professor A. Futaki 
for their useful comments and advice. 
He is also very grateful to Professor R. Goto 
for his advice and encouraging the author. 
He would like to thank the referee for his valuable comments 
and suggesting Theorem~\ref{s1t1.5}. 
This work was partially supported by GCOE `Fostering top leaders in mathematics', 
Kyoto University and by Grant-in-Aid for Young Scientists (B) $\sharp$21740051 from JSPS. 

\begin{center}

\end{center}

\begin{flushright}
\begin{tabular}{l}
\textsc{Takayuki Moriyama}\\
Department of Mathematics\\
Kyoto University\\
Kyoto 606-8502, Japan\\
E-mail: moriyama@math.kyoto-u.ac.jp
\end{tabular}
\end{flushright}

\end{document}